\documentclass[11pt]{article}
\usepackage{amsmath}
\usepackage{amssymb}
\usepackage{amsthm}
\usepackage[pdftex]{graphicx}
\usepackage[margin=10pt,font=small,labelfont=bf]{caption}

\newtheorem{theorem}{Theorem}
\newtheorem{corollary}[theorem]{Corollary}

\newtheorem{proposition}[theorem]{Proposition}
\newtheorem{ex}[theorem]{Example}
\newenvironment{example}{\begin{ex} }{\end{ex}}
\newtheorem{definition}[theorem]{Definition}

\newcommand{\cb}    [1]{\ensuremath{\left  \{      #1  \right \}       }}

\newcommand{\cbgg}  [1]{\ensuremath{\biggl \{      #1  \biggr \}       }}

\newcommand{\of}    [1]{\ensuremath{\left (        #1  \right )        }}
\newcommand{\ofg}   [1]{\ensuremath{\bigl (        #1  \bigr  )        }}

\newcommand{\st} {\ensuremath{|\;}}

\newcommand{\conv}  {{\rm conv \,}}

\newcommand{\cone}{{\rm cone\,}}

\newcommand{\X}{\mathcal{X}}

\renewcommand{\P}{\mathcal{P}}

\renewcommand{\S}{\mathcal{S}}

\newcommand{\R}{\mathbb{R}}

\DeclareMathOperator*{\proj}{proj}

\newcommand{\smz}{\!\setminus\!\{0\}}

\newcommand{\dir}{^{\rm dir}} \newcommand{\poi}{^{\rm poi}}
 
\DeclareMathOperator{\extr}{ext}

\author{Andreas L\"{o}hne \thanks{Friedrich Schiller University Jena,
    Department of Mathematics, 07737 Jena, Germany,
    andreas.loehne@uni-jena.de} \and Benjamin Wei{\ss}ing
  \thanks{Friedrich Schiller University Jena, Department of
    Mathematics,  07737 Jena, Germany,
    benjamin.weissing@uni-jena.de}}

\title{Equivalence between polyhedral projection, multiple objective
  linear programming and vector linear programming}

\begin{document}
\maketitle

\begin{abstract} 
Let a polyhedral convex set be given by a finite number of linear
inequalities and consider the problem to project this set onto a
subspace. This problem, called polyhedral projection problem, is shown
to be equivalent to multiple objective linear programming. The number
of objectives of the multiple objective linear program is by one
higher than the dimension of the projected polyhedron. The result
implies that an arbitrary vector linear program (with arbitrary
polyhedral ordering cone) can be solved by solving a multiple
objective linear program (i.e.\ a vector linear program with the
standard ordering cone) with one additional objective space dimension.
\medskip

\medskip


\end{abstract}

\begin{section}{Problem formulations and solution concepts}
Let $k,n,p$ be positive integers and let two matrices $G \in \R^{k
  \times n}$, $H\in \R^{k \times p}$ and a vector $h \in \R^k$ be
given. We consider the problem of {\em polyhedral projection}, that
is,
\begin{equation}\label{P}
	\tag{PP} \text{compute } Y = \cb{y \in \R^p \st \exists x \in
          \R^n: G x + H y \geq h}\text{.}
\end{equation}
A point $(x,y) \in \R^n \times \R^p$ is said to be {\em feasible} for
\eqref{P} if it satisfies $G x + H y \geq h$. A direction $(x,y) \in
\R^n \times (\R^p\smz)$ is said to be {\em feasible} for \eqref{P} if
it satisfies $G x + H y \geq 0$. A pair $(X\poi,X\dir)$ is said to be
{\em feasible} for \eqref{P} if $X\poi$ is a nonempty set of feasible
points and $X\dir$ is a set of feasible directions.\par
We use $\proj:\R^{n+p}\to\R^{p}$ to denote the projection of a set $X
\subseteq \R^{n+p}$ onto its last $p$ components.  For a nonempty set
$B\subseteq \R^p$, $\conv B$ is the convex hull, and
$\cone B:=\cb{\lambda x \st \; \lambda \geq 0,\; x \in \conv B}$
is the convex cone generated by this set.  We set $\cone \emptyset :=
\cb{0}$.
\begin{definition}\label{def.sltnPP}
A pair $(X\poi,X\dir)$ is called a {\em solution} to \eqref{P} if it
is feasible, $X\poi$ and $X\dir$ are finite sets, and
\begin{equation}\label{eq_repr_k}
Y = \conv \proj X\poi + \cone \proj X\dir\text{.}
\end{equation}
\end{definition}

For positive integers $n,m,q,r$, let $A \in \R^{m\times n}$, $P \in
\R^{q \times n}$, $Z \in \R^{q \times r}$ and $b \in \R^m$ be
given. Consider the {\em vector linear program}
\begin{equation*}\label{VLP}
	\tag{VLP} \text{minimize } Px \text{ s.t. } A x \geq b\text{,}
\end{equation*}
where minimization is understood with respect to the {\em ordering
  cone}
$$C:=\cb{y \in \R^q \st Z^T y \geq 0}.$$ This means that we use the
ordering
\begin{equation}\label{eq_order}
	 w \leq_C y \quad:\Leftrightarrow\quad y-w \in C
         \quad\Leftrightarrow\quad Z^T w \leq Z^T y.
\end{equation}
We assume $ \ker Z^T := \cb{y \in \R^q \st Z^T y = 0} = \cb{0}$, which
implies that the ordering cone $C$ is pointed. Thus, \eqref{eq_order}
defines a partial ordering. If $Z$ is the $q\times q$ unit matrix,
\eqref{VLP} reduces to a {\em multiple objective linear program}. This
special class of \eqref{VLP} is denoted by (MOLP).

A point $x \in \R^n$ is called feasible for \eqref{VLP} if it
satisfies the constraint $A x \geq b$. A direction $x \in \R^n \smz$
is called feasible for \eqref{VLP} if $A x \geq 0$. The feasible set
of \eqref{VLP} is denoted by
$$S:=\cb{x \in \R^n \st A x \geq b}.$$ A pair $(S\poi, S\dir)$ is
called feasible for \eqref{VLP} if $S\poi$ is a nonempty set of
feasible points and $S\dir$ is a set of feasible directions. The
recession cone of $S$ is the set $0^+S = \cb{x \in \R^n \st A x \geq
  0}$. This means that $0^+S\smz$ represents the set of feasible
directions. We refer to a {\em homogeneous problem} (associated to
\eqref{VLP}) if the feasible set $S$ is replaced by $0^+S$.\par
 For a set $X \subseteq \R^n$, we write $P[X]:=\cb{P x \st x \in
  X}$. The set
$$\P:=P[S]+C = \cb{y \in \R^q \st \exists x \in \R^n,\; Z^T y \geq Z^T P x,\; Ax \geq b}$$ is called the {\em upper image} of \eqref{VLP}.  
\begin{definition}\label{def.sltnVLP}
A point $x \in S$ is said to be a {\em minimizer} for \eqref{VLP} if
there is no $v \in S$ such that $P v \leq_C P x$, $P v \neq P x$, that
is,
$$x \in S,\quad P x \not \in P[S] + C\smz.$$ A direction $x \in \R^n
\smz$ of $S$ is called a {\em minimizer} for \eqref{VLP} if the point
$x$ is a minimizer for the homogeneous problem. This can be expressed
equivalently as
$$x \in (0^+ S) \smz,\quad P x \not \in P[0^+ S] + C\smz.$$
If $(S\poi, S\dir)$ is feasible for \eqref{VLP} and the sets $S\poi$,
$S\dir$ are finite; and if
\begin{equation}\label{eq_infatt}
\conv P[S\poi] + \cone P[S\dir] + C = \P,
\end{equation}
then $(S\poi, S\dir)$ is called a {\em finite infimizer} for \eqref{VLP}.\par
  A finite infimizer $(S\poi, S\dir)$ is called a {\em solution} to
  \eqref{VLP} if its two components consist of minimizers only.
\end{definition}
This solution concept for \eqref{VLP} has been introduced in
\cite{Loehne11}. It can be motivated theoretically by a combination of
minimality and infimum attainment established in \cite{HeyLoe11}. Its
relevance for applications has been already discussed indirectly in
earlier papers, see e.g.\ \cite{Dauer87, DauLiu90, Benson98,
  EhrLoeSha12}. The solver {\sc Bensolve} \cite{bensolve, LoeWei16} uses this
concept.
\end{section}

\begin{section}{Equivalence between (PP), (MOLP) and (VLP)}
As a first result, we show that a solution of \eqref{P} can be easily
obtained from a solution of the following multiple objective linear
program
\begin{equation}\label{associated_molp}
	\min \begin{pmatrix} y \\-e^T y \end{pmatrix} \text{ s.t. }  G
        x + H y \geq h,
\end{equation}
where $e := (1,\dots,1)^T$. Both problems \eqref{P} and
\eqref{associated_molp} have the same feasible set but
\eqref{associated_molp} has one additional objective space dimension.

\begin{theorem} \label{th1} Let a polyhedral projection problem
  \eqref{P} be given. If \eqref{P} is feasible, then a solution
  $(S\poi,S\dir)$ (compare Definition \ref{def.sltnVLP}) of the associated
  multiple objective linear program \eqref{associated_molp}
  exists. Every solution $(S\poi,S\dir)$ of \eqref{associated_molp} is
  also a solution of \eqref{P} (Definition \ref{def.sltnPP}).
\end{theorem}
\begin{proof} Since \eqref{P} is feasible, the projected polyhedron $Y$ is nonempty. Thus it has a finite representation. This implies that a solution $\X=(X\poi,X\dir)$ of \eqref{P} exists. We show that  $\X$ is also a solution of the associated multiple objective linear program \eqref{associated_molp}. $\X$ is feasible for \eqref{associated_molp} and its two components are finite sets. By \eqref{eq_repr_k}, we have
	$$ P[S] = \cb{\begin{pmatrix}y\\-e^T y\end{pmatrix} \bigg|\; y
      \in Y} = \conv P[X\poi] + \cone P[S\dir] $$ which implies
  \eqref{eq_infatt} for $C=\R^{p+1}_+$, i.e., $\X$ is a finite
  infimizer.  We have
$$P[S] \subseteq \cb{y \in \R^{p+1} \st e^T y = 0},$$ which implies
  that all points and directions of $S$ are minimizers. Thus $\X$
  consists of minimizers only. We conclude that a solution of
  \eqref{associated_molp} exists.

Let $\S=(S\poi,S\dir)$ be an arbitrary solution of
\eqref{associated_molp}.  By \eqref{eq_infatt} we have
\begin{equation}\label{eq_1}	
	\conv P[S\poi] + \cone P[S\dir] + \R^{p+1} _+= P[S] +
        \R^{p+1}_+,
\end{equation}
where $\R^{p+1} _+$ denotes the nonnegative orthant.  We show that
\begin{equation}\label{eq_2}
\conv P[S\poi] + \cone P[S\dir] = P[S]
\end{equation}
holds. The inclusion $\subseteq$ is obvious by feasibility. Let $y \in
P[S]$, then $e^T y = 0$. By \eqref{eq_1}, $y \in B + \R^{p+1}_+$ for
$B:=\conv P[S\poi] + \cone P[S\dir]$. There is $v \in B$ and $c \in
\R^{p+1}_+$ such that $y = v + c$. Assuming that $c \in
\R^{p+1}_+\smz$ we obtain $e^T c > 0$. But $B \subseteq P[S]$ and
hence the contradiction $0 = e^T y = e^T v + e^T c > 0$. Thus
\eqref{eq_2} holds. Omitting the last components of the vectors
occurring in \eqref{eq_2}, we obtain \eqref{eq_repr_k}, which
completes the proof.
\end{proof}

Of course, (MOLP) is a special case of \eqref{VLP}. In order to obtain
equivalence between \eqref{VLP}, (MOLP) and \eqref{P}, it remains to
show that a solution of \eqref{VLP} can be obtained from a solution of
\eqref{P}. We assign to a given \eqref{VLP} the polyhedral projection
problem
\begin{equation}\label{associated_pp}
	\text{compute } \P = \cb{y \in \R^p \st \exists x \in \R^n :\,
          Z^T y \geq Z^T Px,\; A x \geq b}.
\end{equation}
Obviously, \eqref{VLP} is feasible if and only if
\eqref{associated_pp} is feasible. The next result states that a
solution of \eqref{VLP}, whenever it exists, can be obtained from a
solution of the associated polyhedral projection problem
\eqref{associated_pp}. This result is prepared by the following
proposition. The main idea is that non-minimal points and directions
can be omitted in a certain representation of a nonempty closed convex
set.

\begin{proposition}\label{prop_2}
  Let a nonempty compact set $V \subseteq \R^q$ of points and a
  compact set $R \subseteq \R^q\setminus\cb{0}$ of directions be given
  and define $\P \mathrel{\mathop:}= \conv V + \cone R$. Furthermore,
  let $C \subseteq \R^q$ be a nonempty closed convex cone.  If $\P + C
  \subseteq \P$ and
  \begin{equation}
    \label{eq:cond}
     L \cap C = \cb{0},
  \end{equation}
  where $L:=0^+\P \cap -0^+\P$ is the lineality space of $\P$, then
  \begin{equation*}
    \P = \conv \ofg{V\setminus(\P+C\setminus\cb{0})} + \cone
    \ofg{R\setminus(0^+\P+C\setminus\cb{0})} + C.
  \end{equation*}
\end{proposition}
\begin{proof}
The inclusion $\supseteq$ is obvious. To show the reverse inclusion,
let $ U$ be a linear subspace complementary to $L$.  Since $V$ and $R$
are compact, and $V \neq \emptyset$, $\P$ is a nonempty closed convex
set.  Then the set $\extr(\P \cap U)$ of extreme points of $\P\cap U$
is nonempty, see e.g.\ \cite[Section 2.4]{Gruenbaum03}. An element $v
\in \extr(\P \cap U)$ admits the representation
  \begin{equation*}
    v = \sum_{j\in J}\lambda_j v^j + \sum_{i\in I}\mu_i r^i\text{,}
  \end{equation*}
  with finite index sets $J$ and $I$, $v^j \in V$, $\lambda_j\geq
  0$ for $j\in J$, $\sum_{j\in J} \lambda_j = 1$, and $r^i \in R$,
  $\mu_i \geq 0$ for $i\in I$.  Every $v^j$ and $r^i$ may be
  decomposed by means of $v^j = v^j_ U + v^j_ L$ and $r^i = r^i_ U +
  r^i_ L$ with $v^j_ U,r^i_ U\in U$ and $v^j_ L,r^i_ L\in L$.
  Therefore,
  \begin{equation*}
    v = \sum_{j \in J}\lambda_j v^j_ U +\sum_{i \in I}\mu_i r^i_ U
    +\underbrace{ \sum_{j \in J}\lambda_j v^j_ L + \sum_{i \in I}\mu_i
      r^i_ L }_{{}\in L}\text{.}
  \end{equation*}
  Because $v\in U$, the last two sums vanish.  In the resulting
  representation
  \begin{equation*}
    v = \sum_{j \in J}\lambda_j v^j_ U + \underbrace{\sum_{i \in
        I}\mu_i r^i_ U}_{{}=0}
  \end{equation*}
  the second sum equals to zero because $v$ is an extremal point.  For
  the same reason, $v^j_ U = v$ for every $j$ with $\lambda_j>0$
  follows.  Hence, $v^j = v + v^j_ L$.
  
  Assume that $v^j \in \P + C\setminus\cb{0}$.  There exist $y\in \P$
  and $c\in C\setminus\cb{0}$ such that $v^j = y + c$.  With
  decompositions $y=y_ L + y_ U$ and $c=c_ L + c_ U$, where $c_ U \neq
  0$ (because of condition \eqref{eq:cond} and $c\neq 0$), this leads
  to
    $$ v + v^j_ L = v_j = y_ U + y_ L + c_ U + c_ L\text{,}$$ whence
   $$ v = y_ U + c_ U + \underbrace{y_ L + c_ L - v^j_
    L}_{{}=0}\text{.}$$ Again, the last part being zero results from
  $v$ being an element of $ U$.  Let $\mu \in (0,1)$ be given.  As a
  linear combination of elements of $ U$, $y_ U + \frac{1}{\mu} c_ U$
  belongs to $U$.  On the other hand,
  $$y_ U + \frac{1}{\mu} c_ U = \underbrace{y+\frac{1}{\mu}
    c}_{{}\in\P} \underbrace{-y_ L -\frac{1}{\mu} c_ L}_{{}\in L} \in
  \P.$$ Hence, $y_ U+\frac{1}{\mu} c_ U \in \P\cap U$.  By adding a
  meaningful zero, a representation of $v$,
  \begin{align*}
    v &= y_ U + c_ U\\ &= \mu \left(y_ U + \frac{1}{\mu}c_ U \right) +
    (1-\mu)y_ U\text{,}
  \end{align*}
  as a convex combination of different points of $\P\cap U$ is found,
  which contradicts $v$ being an extremal point.  Thus, $v^j \notin \P
  + C\setminus\cb{0}$; and altogether
  \begin{equation}
    \label{eq:vjmin}
    \extr(\P\cap U) \subseteq V\setminus(\P+C\setminus\cb{0}) + L.
  \end{equation}
  \par Now an extremal direction $r\in 0^+{\left(\P\cap U\right)}$ is
  considered. Since $V$ and $R$ are compact, $\P = \conv V + \cone R$
  implies $0^+\P = \cone R$ (see e.g.\ \cite[Corollary
    9.1.1]{Rockafellar72}). Thus, a representation
  \begin{equation*}
    r = \sum_{j \in J}\mu_j r^j_ U + \underbrace{\sum_{j \in J}\mu_j
      r^j_ L}_{{}=0}
  \end{equation*}
  by $r^j = r^j_U + r^j_L \in R$ can be obtained.  From extremality of
  $r$ it follows that all $r^j_U\neq 0$ with $\mu_j > 0$ coincide with
  $r$ up to positive scaling.  Such an $j$ is taken and, without loss
  of generality, one may assume $\mu_j=1$, i.e.:
  \begin{equation*}
    r^j = r+r^j_L\text{.}
  \end{equation*}
  We show that
  \begin{equation}\label{eq:d4}
   r^j \in \left(R\setminus \left(0^+\P+C\setminus\cb{0}\right)
   \right) \cup (C+L).
   \end{equation}
  Indeed, let $r^j\in 0^+\P+C\setminus\cb{0}$, that is, $r^j = y^h +
  c$ for $y^h \in 0^+\P$ and $c\in C\setminus\cb{0}$. We need to show
  that $r^j \in C+L$ follows. If $y^h\in L$, this is obvious. Thus,
  let $y^h\not\in L$. Consider the decompositions $y^h = y^h_U +y^h_L$
  and $c = c_U+c_L$ with $y^h_U,c_U\in U$, $y^h_L,c_L\in L$,
  $y^h_U\neq 0$ and, by \eqref{eq:cond}, $c_U \neq 0$. We obtain
   $$ r+r^j_L = r^j = y^h_U+y^h_L + c_U + c_L \text{.}$$ Since $r \in
  U$,
   $$ r = y^h_U + c_U + y^h_L+c_L-r^j_L= y^h_U + c_U\text{.}$$ From
  $L\subseteq 0^+\P$ the inclusion $0^+\P+L\subseteq 0^+\P$ and
  subsequently $y^h_U \in 0^+\P - \cb{y^h_L}\subseteq 0^+\P$ is
  deduced.  Moreover, $c_U\in C+L \subseteq C + 0^+\P \subseteq
  0^+\P$.  Noted that $0^+(\P\cap U) = 0^+\P\cap U$, both directions
  $y^h_u$ and $c_U$ are in $0^+(\P\cap U)$.  Therefore the
  representation of $r = y^h_U + c_U$ as conic combination of elements
  of $0^+(\P\cap U)$, in which $r$ was supposed to be extremal, proves
  equality of $r,y^h_U$ and $c_U$ up to positive scaling.  This
  implies $r\in C+L$ and $r^j \in C+L$.
  
  From \eqref{eq:d4} we deduce
  $$ r^j \in \cone
  \left(R\setminus\left(0^+\P+C\setminus\cb{0}\right)\right)+C + L$$
  and hence
  \begin{equation}
    \label{eq:rimin}
      0^+(\P\cap U)\subseteq \cone\of{R\setminus(0^+\P +
        C\setminus\cb{0})} + C + L.
  \end{equation}
  \par Any lineality direction $l\in L$ can be represented by a conic
  combination of elements of $R$: $l = \sum_{j\in J} \mu_j l^j$.  For
  any $l^k$ with $\mu_k >0$ this results in
  \begin{equation*}
    l^k = \frac{1}{\mu_k}\left(l - \sum_{j \in J\setminus\cb{k}}\mu_j
    l^j\right)\in - 0^+\P
  \end{equation*}
  and therefore in $l^k \in L$.  If such an $l^k$ is an element of $
  0^+\P+C$, then there exist $r\in 0^+\P$ and $c\in C$ with $l^k = r +
  c$.  This implies
  \begin{equation*}
    c = \underbrace{l^k-r}_{{}\in - 0^+\P} \in L\text{,}
  \end{equation*}
  and by \eqref{eq:cond}, $c=0$ follows. Therefore, $l^k\in
  R\setminus(0^+\P+C\setminus\cb{0})$, resulting in
  \begin{equation}
    \label{eq:lnspc}
     L\subseteq\cone\of{R\setminus(0^+\P+C\setminus\cb{0})}.
  \end{equation}\par
  Combining the results \eqref{eq:lnspc} and \eqref{eq:rimin} yields
  \begin{align*}
     L + 0^+ (\P\cap U) &\subseteq L + \cone\of{R\setminus(0^+\P +
       C\setminus\cb{0})} + C + L\\ &\subseteq
     \cone\of{R\setminus(0^+\P + C\setminus\cb{0})} + C\text{.}
  \end{align*}
Using \eqref{eq:vjmin} and the decomposition $\P= L + \left[\P\cap
  U\right]$, we obtain
  \begin{align*}
    \P &= L + (\P\cap U)\\ &= L + \of{\conv\extr(\P\cap U) +
      0^+(\P\cap U)}\\ &\subseteq
    \conv\of{V\setminus(\P+C\setminus\cb{0})}
    +\cone\of{R\setminus(0^+\P+C\setminus\cb{0})} + C,
  \end{align*}
  which proves the claim.
\end{proof}

\begin{theorem} \label{th2}
Let a vector linear program \eqref{VLP} be given. If \eqref{VLP} is
feasible, a solution of the associated polyhedral projection problem
\eqref{associated_pp} according to Definition \ref{def.sltnPP} exists. Let $\X=(X\poi,X\dir)$ be a solution of
\eqref{associated_pp}. Assume that \eqref{eq:cond} is satisfied and
set
	$$ S\poi := \cb{ x \in \R^n \st (x,y) \in X\poi,\; y \not\in
  \P + C\smz}, $$
	$$ S\dir := \cb{x \in \R^n \st (x,y) \in X\dir,\; y \not\in
  0^+ \P + C\smz}.$$ Then $(S\poi,S\dir)$ is a solution  of
\eqref{VLP} in the sense of Definition \ref{def.sltnVLP}.  Otherwise,
if \eqref{eq:cond} is violated, \eqref{VLP} has no solution.
\end{theorem}
\begin{proof}
The existence of a solution of the polyhedral projection problem
\eqref{associated_pp} is evident, because a polyhedron has a finite
representation. Consider a solution $(X\poi,X\dir)$ of
\eqref{associated_pp}. From Proposition \ref{prop_2}, we obtain
	$$ \P = \conv P[S\poi] + \cone P[S\dir] + C.$$ Hence,
$(S\poi,S\dir)$ is a finite infimizer for \eqref{VLP}. It is evident
that $S\poi$ and $S\dir$ consist of minimizers only.

Assume now that \eqref{eq:cond} is violated. Take $y \in L\cap C\smz$,
then for any $x\in S$, $Px = Px - y + y$, where $Px-y \in \P$ and $y
\in C\smz$. Thus, $x$ is not a minimizer.
\end{proof}

The following statement is an immediate consequence of Theorem
\ref{th2}.

\begin{corollary}
A solution for \eqref{VLP} exists if and only if \eqref{VLP} is
feasible and \eqref{eq:cond} is satisfied.
\end{corollary}

In the remainder of this paper we show how a solution of \eqref{VLP}
can be obtained from an {\em irredundant solution} of the associated
projection problem \eqref{P}. A solution $(X\poi, X\dir)$ of \eqref{P}
is called {\em irredundant} if there is no solution $(V\poi,V\dir)$ of
\eqref{P} satisfying
$$ V\poi \subseteq X\poi,\quad V\dir \subseteq X\dir, \quad
(V\poi,V\dir) \neq (X\poi,X\dir).$$ The computation of an irredundant
solution of \eqref{P} from an arbitrary solution of \eqref{P} does not
depend on the dimension $n$. It can be realized, for instance, by
vertex enumeration in $\R^p$.

\begin{theorem} \label{th3}
Let a vector linear program \eqref{VLP} be given. If \eqref{VLP} is
feasible, an irredundant solution of the associated polyhedral
projection problem \eqref{associated_pp} exists. Let
$\X=(X\poi,X\dir)$ be an irredundant solution of
\eqref{associated_pp}. Assume that \eqref{eq:cond} is satisfied and
set
	$$ S\poi := \cb{ x \in \R^n \st (x,y) \in X\poi}, $$
	$$ S\dir := \cb{x \in \R^n \st (x,y) \in X\dir,\; y \not\in L
  + C\smz}.$$ Then $(S\poi,S\dir)$ is a solution of
\eqref{VLP}. Otherwise, if \eqref{eq:cond} is violated, \eqref{VLP}
has no solution.
\end{theorem}
\begin{proof}
By Theorem \ref{th2}, it remains to show that $(S\poi,S\dir)$ consists
of minimizers only.  Assume that $x$ for $(x,y)\in X\poi$ is not a minimizer.
There exists $z \in \P$ and $c \in C\smz$ such that $y = z + c$. Let
us denote the elements of $X\poi$ as
	$$ X\poi = \cb{(x^1,y^1), \dots, (x^{\alpha-1}, y^{\alpha
    -1}),(x, y)}.$$ The point $z$ can be represented by $X\poi$ and $d \in 0^+\P$, which
yields
$$ y - c = z = \sum_{i=1}^{\alpha-1} \lambda_i y^i + \lambda y + d, \qquad
\lambda_1,\dots,\lambda_{\alpha-1},\lambda \geq 0,\quad
\sum_{i=1}^{\alpha -1} \lambda_i + \lambda = 1.$$ We set $v :=
\sum_{i=1}^{\alpha-1} \lambda_i y^i \in \P$ and consider two cases:
(i) For $\lambda = 1$, we have $v=0$ and hence $c = -d$, a
contradiction to \eqref{eq:cond}. (ii) For $\lambda < 1$, we obtain
\begin{align*}
 y &= \sum_{i=1}^{\alpha -1} \frac{\lambda_i}{1-\lambda} y^i +
 \frac{1}{1-\lambda} (c+d) \in \conv\cb{y^1,\dots, y^{\alpha - 1}} +
 \cone\proj X\dir \\ &= \conv \of{( \proj X\poi) \setminus\cb{y}} +
 \cone\proj X\dir.
 \end{align*}
This contradicts the assumption that the solution $(X\poi,X\dir)$ is
irredundant.

	Assume now that $(x,y)\in X\dir$ with $y \not\in L+C\smz$ is
        not a minimizer. There exists $z \in 0^+\P$ and $c \in C\smz$
        such that $y = z + c$. Let us denote the elements of $X\dir$
        by
	$$ X\dir = \cb{ (x^1, y^1) , \dots, (x^{\beta-1}, y^{\beta
            -1}), (x , y)}.$$ The direction $z$ can be represented by
        $X\dir$, which yields
$$ y - c = z = \sum_{i=1}^{\beta-1} \lambda_i y^i + \lambda y, \qquad
        \lambda_1,\dots,\lambda_{\beta-1},\lambda \geq 0.$$ We set $v
        := \sum_{i=1}^{\beta-1} \lambda_i y^i \in 0^+\P$ and
        distinguish two cases: (i) Let $\lambda \geq1$. Then
        $(1-\lambda) y - v \in -0^+\P \cap C\smz$. This contradicts
        \eqref{eq:cond} since $C \subseteq 0^+\P$. (ii) Let $\lambda <
        1$. Then
$$ y = w + d \quad \text{ for } \quad \gamma_i :=
        \frac{\lambda_i}{1-\lambda} \geq 0, \quad w :=
        \sum_{i=1}^{\beta-1} \gamma_i y^i \quad \text{ and }\quad
        d:=\frac{1}{1-\lambda} c \in C\smz.$$ We have $w \not\in
        -0^+\P$ since otherwise $y \in L + C\smz$ would follow. For
        $d$ there is a representation
$$ d = \sum_{i=1}^{\beta -1} \mu_i y^i + \mu y \qquad
        \mu_1,\dots,\mu_{\beta-1},\mu \geq 0.$$ The condition $\mu
        \geq 1$ would imply $ w = (1-\mu) y - d \in -0^+\P$. Thus we
        have $\mu < 1$ and hence
$$ y = \sum_{i=1}^{\beta -1} \frac{\gamma_i + \mu_i}{1-\mu} y^i \in
        \cone\cb{y^1,\dots, y^{\beta - 1}} = \cone \of{( \proj X\dir)
          \setminus\cb{y}}.$$ This contradicts the assumption that the
        solution $(X\poi,X\dir)$ is irredundant.
\end{proof}
\end{section}

\begin{section}{Examples and remarks}
  It is well known that \eqref{VLP} can be solved by considering the
  multiple objective linear program
  \begin{equation}\label{eq.molp_clssc}
    \text{minimize } Z^TPx \text{ s.t. } A x \geq b\text{,}
  \end{equation}
  compare, c.f., \cite{SawNakTan85}.  The $r$ columns of the matrix
  $Z\in\R^{q\times r}$ correspond to the defining inequalities of $C$ (or equivalently, to the generating vectors of the dual cone).  The objective space dimension $r$ of
  \eqref{eq.molp_clssc} can be much larger than the objective space dimension $q$ of the initial vector linear program, see e.g. \cite{Roux2015} for a sample application and \cite{LoeRud15} for the number of inequalities required to describe the ordering cone there. In contrast to \eqref{eq.molp_clssc}, with our approach, the objective space dimension is increased only by one.\par

  In the following toy-example with $q=2$ we illustrate the procedure.    
   
  \begin{example}\label{ex:8}
    
Let us consider the following instance of \eqref{VLP}:
    \begin{equation}\label{eq.exmpl_vlp_orig}
      \text{minimize}\;      
      \begin{pmatrix}
         1 & -1\\
         1 & 1
      \end{pmatrix}
      x\quad \text{s.t.}\;\;x\in S\text{,}
    \end{equation}
    where the feasible set $S$ is defined by
    \begin{equation}
    S\mathrel{\mathop:}=\cb{x\in\R^n \bigg|
      \begin{pmatrix}
        1&0\\
        1&-1\\
        1&1
      \end{pmatrix}
      x \geq
      \begin{pmatrix}
        0\\-1\\-1
      \end{pmatrix}}
      \text{.}
    \end{equation}
    Minimization is understood with respect to the partial
    ordering generated by the ordering cone
    \begin{equation*}
      C\mathrel{\mathop:}=\cbgg{y\in\R^q\,\bigg|\,
        \underbrace{\begin{pmatrix}
          -1 & 2\\
          2 & 1
        \end{pmatrix}}_{\mathrel{=\mathop{:}}Z^T}y\geqslant 0
      }\text{.}
    \end{equation*}    
    The solution concept for the vector linear program \eqref{eq.exmpl_vlp_orig} demands the computation of a
    representation of the
      upper image $\P=P[S]+C$, which is depicted in Figure \ref{fig.img}.
    \begin{figure}[ht]
      \includegraphics[keepaspectratio=true,width=\textwidth]{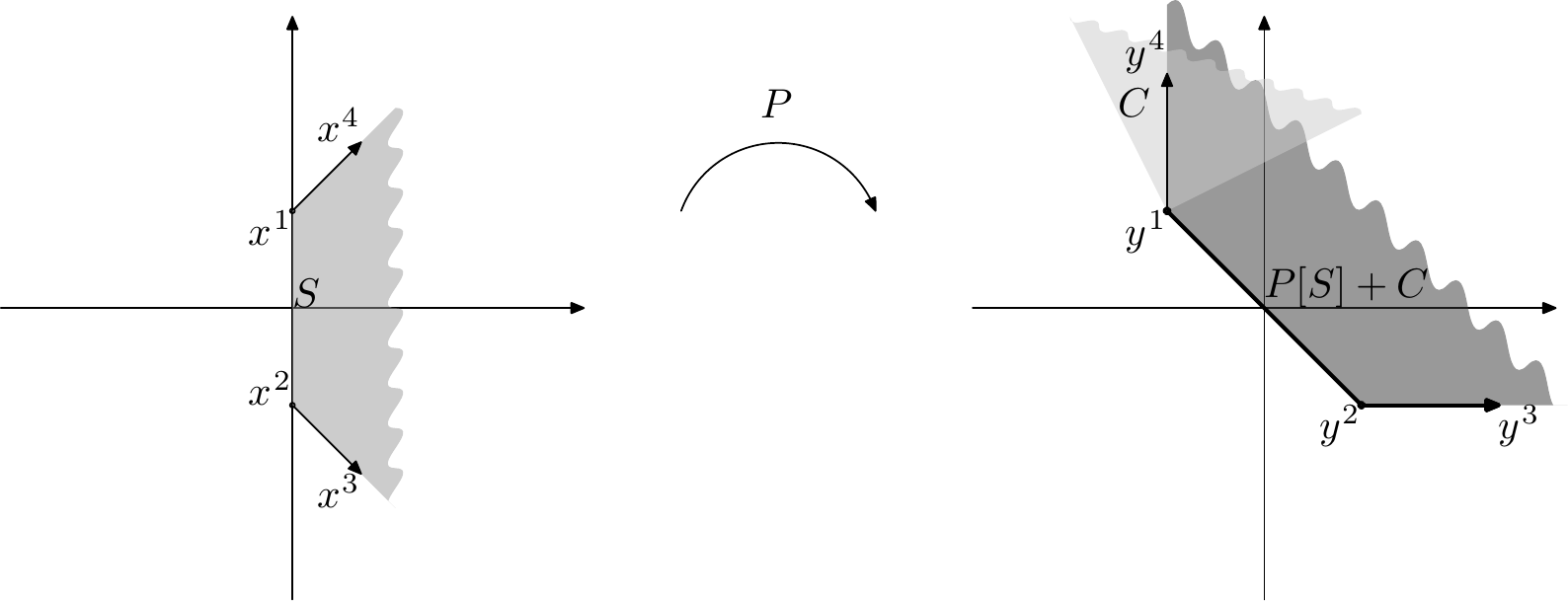}
      \caption{The feasible set $S$ and the upper image $P[S]+C$
        of Example \ref{ex:8}.  A solution is given by the feasible
        points $x^1$,$x^2$ and the feasible direction $x^3$.  Their
        respective image-vectors $y^1$,$y^2$ and $y^3$ generate the
        upper image.  It can be seen that $x^4$ is not part of a
        solution, as the image-vector $y^4$ belongs to the ordering
        cone $C$ and is therefore not a minimal direction.\label{fig.img}}
       
    \end{figure}
    First, we express $\P$ as an instance of \eqref{P}.  This step pushes the ordering cone into the
    constraint set, compare \eqref{associated_pp}:
    \begin{equation}\label{eq.pp_assoc}
      \text{compute}\;\P = \cb{y\in\R^q\,|\, \exists x\in S, Z^Ty\leq Z^TPx
      }\text{,}
    \end{equation}
    Now we formulate the corresponding instance of (MOLP), with one additional image space dimension, compare
    \eqref{associated_molp}:
    \begin{equation}\label{eq.ex_assoc_molp}
      \text{minimize}\,\underbrace{
      \begin{pmatrix}
        0 & 0 & 1 & 0\\
        0 & 0 & 0 & 1\\
        0 & 0 & -1 & -1
      \end{pmatrix}}_{{}=\mathrel{\mathop:}\hat{P}}
      \begin{pmatrix}
        x\\y
      \end{pmatrix}
      \quad
      \text{s.t.}\;
      \begin{pmatrix}
        x\\
        y
      \end{pmatrix}
      \in \hat{S}\text{,}
    \end{equation}
    where the feasible set $\hat{S}$ is the same as in
    \eqref{eq.pp_assoc}, i.e.
    \begin{equation*}
      \hat{S} {\mathrel{\mathop:}=}\cb{
      \begin{pmatrix}
        x\\
        y
      \end{pmatrix}\in\R^{n+q}\bigg|
        \begin{pmatrix}
          1 & 0\\
          1 & -1\\
          1 & 1\\
          -1& -3\\
          -3 & 1
        \end{pmatrix}
        x
        +
        \begin{pmatrix}
          0 & 0\\
          0 & 0\\
          0 & 0\\
          -1 & 2\\
          2 & 1
        \end{pmatrix}
        y
        \geqslant
        \begin{pmatrix}
          0\\
          -1\\
          -1\\
          0\\
          0
        \end{pmatrix}}\text{.}
    \end{equation*}
     Now we consider a solution
    $(\hat{S}\poi,\hat{S}\dir)$ to \eqref{eq.ex_assoc_molp}, which consists of $\hat{S}\poi\mathrel{\mathop:}=\cb{\hat{y}^1,\hat{y}^2}$ with feasible points
    \begin{equation*}
      \hat{y}^1 = 
      \begin{pmatrix}
        x^1\\y^1
      \end{pmatrix}=
      \begin{pmatrix}
        0\\ 1 \\-1\\1
      \end{pmatrix}\text{,}\qquad
      \hat{y}^2 = 
      \begin{pmatrix}
        x^2\\y^2
      \end{pmatrix}=
      \begin{pmatrix}
        0\\-1\\1\\-1
      \end{pmatrix}\text{;}
    \end{equation*}
    and $\hat{S}\dir\mathrel{\mathop:}=\cb{\hat{y}^3,\hat{y}^4}$ with
    the feasible directions
    \begin{equation*}
      \hat{y}^3 = 
      \begin{pmatrix}
        x^3\\y^3
      \end{pmatrix}=
      \begin{pmatrix}
        1\\-1\\ 2\\0
      \end{pmatrix}\text{,}\qquad
      \hat{y}^4 = 
      \begin{pmatrix}
        x^4\\y^4
      \end{pmatrix}=
      \begin{pmatrix}
        0\\0\\-1\\2
      \end{pmatrix}\text{.}
    \end{equation*}
    This means $\hat{y}^i$ is $\R^3_+$-minimal for $i=1,\ldots,4$ and
    \begin{equation*}
      \hat{P}[\hat{S}]+\R^3_+ = \conv\hat{P}[\hat{S}\poi]+\cone\hat{P}[\hat{S}\dir]+\R^3_+\text{.}
    \end{equation*}
	From Theorem \ref{th1} we deduce that $(\hat{S}\poi,\hat{S}\dir)$
	is also a  solution for the polyhedral projection problem \eqref{eq.pp_assoc}. This solution is irredundant, so the $x$-components of the points $\hat{y}^1,\hat{y}^2$, that is $x^1$ and $x^2$, are the points of a solution to the original vector linear program \eqref{eq.exmpl_vlp_orig} by
    Theorem \ref{th3}.
    It remains to sort out those directions whose $y$-part belongs to
    the ordering cone $C$ (compare Theorem \ref{th3}):  This is the case for the direction
    $\hat{y}^4$, as $Z^Ty^4=
    (5,0)^T \geqslant 0
    $.
    Thus the solution for \eqref{eq.exmpl_vlp_orig} consists of the
    feasible points $x^1,x^2$ and the feasible direction $x^3$ (also    compare Figure \ref{fig.img}):
    \begin{equation*}
      S\poi=\cb{
        \begin{pmatrix}
      0\\1
        \end{pmatrix},
        \begin{pmatrix}
      0\\-1
        \end{pmatrix}
      }\qquad
      S\dir=\cb{
        \begin{pmatrix}
          1\\-1
        \end{pmatrix}
      }\text{.}
    \end{equation*}
    It generates the upper image of \eqref{eq.exmpl_vlp_orig} (Figure
    \ref{fig.img}) by means of
    \begin{equation*}
      \P = \conv P[S\poi]+\cone P[S\dir]+C\text{.}
    \end{equation*}
	Finally, let us demonstrate how $\P$ can be obtained directly from $\hat{\P}$. We start with an irredundant representation  
	$$ \hat{\P} = \conv\cb{\bar y^1, \bar y^2} + \cone\cb{\bar y^3, \bar y^4, \bar y^5},$$
	where
	$$ \bar y^1=\begin{pmatrix}
		-1\\1\\ 0
	\end{pmatrix},\;\bar y^2=\begin{pmatrix}
		1\\-1\\ 0
	\end{pmatrix},\; \bar y^3=\begin{pmatrix}
		2\\0\\-2
	\end{pmatrix},\;\bar y^4=\begin{pmatrix}
		-1\\2\\ -1
	\end{pmatrix},\; \bar y^5=\begin{pmatrix}
		0\\0\\ 1
	\end{pmatrix}. $$
	We rule out those $\bar y^i$ where the condition $e^T \bar y^i = 0$ is violated, whence $\bar y^5$ is cancelled. Then we delete the last component of each vector and obtain	
		$$ \P= \conv\cb{y^1, y^2} + \cone\cb{y^3,y^4}$$
		with 
		$$ y^1=\begin{pmatrix}
			-1\\1
		\end{pmatrix},\; y^2=\begin{pmatrix}
			1\\-1
		\end{pmatrix},\;  y^3=\begin{pmatrix}
			2\\0
		\end{pmatrix},\; y^4=\begin{pmatrix}
			-1\\2
		\end{pmatrix}. $$	
  \end{example}
  
  In the preceeding example one needs one additional objective, wheras the ``classical'' approach
      \eqref{eq.molp_clssc} does {\em not} require any additional objective. Note that every step of the procedure presented here is independent of the actual values of  $q\geqslant 2$ and $r\geqslant q$. In the following example we consider the case $r > q+1$. This leads to an advantage in comparison to the classical method \eqref{eq.molp_clssc}. The cone $C \subseteq \R^3$ of the vector linear program \eqref{eq:vlp9} has $6$ extreme directions and a solution is obtained from a solution of a corresponding multiple objective linear program \eqref{eq:molp9} with only 4 objectives. Note that in the classical approach, see \eqref{eq.molp_clssc}, 6 objectives are required.
  \begin{example}\label{ex:9} Consider the vector linear program
	  \begin{equation}\label{eq:vlp9}
	  	\text{\rm minimize } Px \text{ s.t. } B x \geq a, \; x \geq 0
	  \end{equation} 
with ordering cone $C = \cb{y \in \R^3 \st Z^T y \geq 0}$ and data
$$ P=\begin{pmatrix}
 1& 0&-1\\ 
 1& 1& 0\\ 
 0& 1& 1\\ 
\end{pmatrix}\!,\, B = \begin{pmatrix}
 1& 1& 1\\ 
 1& 2& 2\\ 
 2& 2& 1\\ 
 2& 1& 2\\ 
\end{pmatrix}\!,\, a = \begin{pmatrix}
 3\\ 
 4\\ 
 4\\ 
 4\\ 
\end{pmatrix}\!,\, Z = \begin{pmatrix}
 4& 2& 4& 1& 0& 0\\ 
 2& 4& 0& 0& 1& 4\\ 
 2& 2& 2& 2& 2& 2\\ 
\end{pmatrix}\!.$$
We assign to \eqref{eq:vlp9} the multiple objective linear program
\begin{equation}\label{eq:molp9}
	 \min \hat P \begin{pmatrix} x\\ y \end{pmatrix} \text{ s.t. } \hat B \begin{pmatrix} x\\ y \end{pmatrix} \geq \hat a, \; x \geq 0
\end{equation}
with objective function 
$$\hat P \begin{pmatrix} x\\ y \end{pmatrix} = \begin{pmatrix} y\\ -e^T y \end{pmatrix}$$
and constraints $B x \geq a$, $Z^T y \geq Z^T P x$, $x \geq 0$. Thus, the data of \eqref{eq:molp9} are 
$$ \hat P = \begin{pmatrix}
 0& 0& 0& 1& 0& 0\\ 
 0& 0& 0& 0& 1& 0\\ 
 0& 0& 0& 0& 0& 1\\ 
 0& 0& 0&-1&-1&-1\\ 
\end{pmatrix}\!,\, \hat B = \begin{pmatrix}
 1& 1& 1& 0& 0& 0\\ 
 1& 2& 2& 0& 0& 0\\ 
 2& 2& 1& 0& 0& 0\\ 
 2& 1& 2& 0& 0& 0\\ 
-6&-4& 2& 4& 2& 2\\ 
-6&-6& 0& 2& 4& 2\\ 
-4&-2& 2& 4& 0& 2\\ 
-1&-2&-1& 1& 0& 2\\ 
-1&-3&-2& 0& 1& 2\\ 
-4&-6&-2& 0& 4& 2\\ 
\end{pmatrix}\!,\,
\hat a = \begin{pmatrix}
 3\\ 
 4\\ 
 4\\ 
 4\\ 
 0\\ 
 0\\ 
 0\\ 
 0\\ 
 0\\ 
 0\\ 
\end{pmatrix}\!.$$
A solution to \eqref{eq:molp9} consists of $\hat S\poi = \cb{\hat y^1, \hat y^2, \hat y^3}$ with 
$$
	 \hat y^1 = \begin{pmatrix} x^1\\y^1 \end{pmatrix} = 
	 \begin{pmatrix}
	  2\\ 
	  0\\ 
	  1\\ 
	  1\\ 
	  2\\ 
	  1\\ 
	 \end{pmatrix},\,
	 \hat y^2 = \begin{pmatrix} x^2\\y^2 \end{pmatrix} = 
	 \begin{pmatrix}
	  1\\ 
	  0\\ 
	  2\\ 
	 -1\\ 
	  1\\ 
	  2\\ 
	 \end{pmatrix},\,
	 \hat y^3 = \begin{pmatrix} x^3\\y^3 \end{pmatrix} = 
	 \begin{pmatrix}
	  0\\ 
	  0\\ 
	  4\\ 
	 -4\\ 
	  0\\ 
	  4\\ 
	 \end{pmatrix}
	 $$
and $\hat S\dir = \cb{\hat y^4,\dots,\hat y^8}$, (again $\hat y^i = (x^i,y^i)^T$) with 
$$
	 \hat y^4=
	 \begin{pmatrix}
	  0\\ 
	  0\\ 
	  0\\ 
	  0\\ 
	 -1\\ 
	  2\\ 
	 \end{pmatrix}\!,\;
	 \hat y^5=
	 \begin{pmatrix}
	  0\\ 
	  0\\ 
	  0\\ 
	  2\\ 
	  2\\ 
	 -1\\ 
	 \end{pmatrix}\!,\;
	 \hat y^6=
	 \begin{pmatrix}
	  0\\ 
	  0\\ 
	  1\\ 
	 -1\\ 
	  0\\ 
	  1\\ 
	 \end{pmatrix}\!,\;
	 \hat y^7=
	 \begin{pmatrix}
	  0\\ 
	  0\\ 
	  0\\ 
	  1\\ 
	  0\\ 
	  0\\ 
	 \end{pmatrix}\!,\;
	 \hat y^8=
	 \begin{pmatrix}
	  0\\ 
	  0\\ 
	  0\\ 
	  0\\ 
	  1\\ 
	  0\\ 
	 \end{pmatrix}\!.$$
	 One can easily check that $y^4, y^5, y^7, y^8 \in C$ and $y^6 \not \in C$. Thus, a solution to the VLP \eqref{eq:vlp9} consists of
	 $S\poi = \cb{x^1, x^2, x^3}$ and $S\dir = \cb{x^6}$, that is,
	 $$ S\poi = \cb{
	 	 \begin{pmatrix}
	 	  2\\ 
	 	  0\\ 
	 	  1\\ 
	 	 \end{pmatrix}
	 	 \begin{pmatrix}
	 	  0\\ 
	 	  0\\ 
	 	  4\\ 
	 	 \end{pmatrix}
	 	 \begin{pmatrix}
	 	  1\\ 
	 	  0\\ 
	 	  2\\ 
	 	 \end{pmatrix}},\quad S\dir = \cb{\begin{pmatrix} 0\\ 
	  0\\ 
	  1\\ 	 
	 \end{pmatrix}}
	 $$
The upper image $\hat \P$ of the MOLP \eqref{eq:molp9} is given by its vertices
$$
	 \bar y^1 =
	 \begin{pmatrix}
	  1\\ 
	  2\\ 
	  1\\ 
	 -4\\ 
	 \end{pmatrix}\!,\,
	 \bar y^2 = 
	 \begin{pmatrix}
	 -1\\ 
	  1\\ 
	  2\\ 
	 -2\\ 
	 \end{pmatrix} \!,\,
	 \bar y^3 = 
	 \begin{pmatrix}
	 -4\\ 
	  0\\ 
	  4\\ 
	  0\\ 
	 \end{pmatrix}
	 $$
	 and its extreme directions
	 $$
	 \bar y^4 = 
	 \begin{pmatrix}
	  0\\ 
	 -1\\ 
	  2\\ 
	 -1\\ 
	 \end{pmatrix}\!,\,
	 \bar y^5 = 
	 \begin{pmatrix}
	  2\\ 
	  2\\ 
	 -1\\ 
	 -3\\ 
	 \end{pmatrix}\!,\,
	 \bar y^6 = 
	 \begin{pmatrix}
	 -1\\ 
	  0\\ 
	  1\\ 
	  0\\ 
	 \end{pmatrix},$$
	 $$
	 \bar y^7 = 
	 \begin{pmatrix}
	  1\\ 
	  0\\ 
	  0\\ 
	 -1\\ 
	 \end{pmatrix}\!,\,
	 \bar y^8 = 
	 \begin{pmatrix}
	 	  0\\ 
	 	  1\\ 
	 	  0\\ 
	 	 -1\\ 
	 	 \end{pmatrix}
	 \!,\,
	 \bar y^9 = 
	 \begin{pmatrix}
	  0\\ 
	  0\\ 
	  0\\ 
	  1\\ 
	 \end{pmatrix}
	 .	 
	 $$
	 We sort out $\bar y^9$, as $e^T \bar y^9 \neq 0$, and we delete the last component of each vector $\bar y^1, \dots, \bar y^8$. As a result we obtain the vertices
	 $$
	 	 y^1 =
	 	 \begin{pmatrix}
	 	  1\\ 
	 	  2\\ 
	 	  1\\ 	 	 
	 	 \end{pmatrix}\!,\,
	 	 y^2 = 
	 	 \begin{pmatrix}
	 	 -1\\ 
	 	  1\\ 
	 	  2\\ 	 	 
	 	 \end{pmatrix} \!,\,
	 	 y^3 = 
	 	 \begin{pmatrix}
	 	 -4\\ 
	 	  0\\ 
	 	  4\\ 	 	   
	 	 \end{pmatrix}
	 	 $$
		 and the extreme directions
		 $$
		 y^4 = 
		 \begin{pmatrix}
		  0\\ 
		 -1\\ 
		  2\\  
		 \end{pmatrix}\!,\,
		 y^5 = 
		 \begin{pmatrix}
		  2\\ 
		  2\\ 
		 -1\\ 
		 \end{pmatrix}\!,\,
		 y^6 = 
		 \begin{pmatrix}
		 -1\\ 
		  0\\ 
		  1\\
		 \end{pmatrix}\!,\,
		 y^7 = 
		 \begin{pmatrix}
		  1\\ 
		  0\\ 
		  0\\
		 \end{pmatrix}\!,\,
		 y^8 = 
		 \begin{pmatrix}
		 	  0\\ 
		 	  1\\ 
		 	  0\\ 
		 	 \end{pmatrix}	 	 
		 $$	 
	 of the upper image $\P$ of the VLP \eqref{eq:vlp9}.	 
  \end{example}
  
  In the next example we have used the VLP solver {\sc bensolve}
  \cite{bensolve, LoeWei16} in order to compute the image of a linear map over a
  polytope, which can be expressed as a polyhedral projection
  problem. The solver is not able to handle ordering cones
  $C=\cb{0}$. Therefore a transformation into (MOLP) with one
  additional objective space dimension is required.\par
  
\begin{example}\label{ex1}
	Let $W$ be the $729$-dimensional unit hypercube and let $P$ be
        the $3\times 729$ matrix whose columns are the $9^3=729$
        different ordered arrangements of $3$ numbers out of the set
        $\cb{-4,-3,-2,-1,0,1,2,3,4}$. The aim is to compute the
        polytope $P[W]:=\cb{Pw\st w \in W}$. To this end, we consider
        the multiple objective linear program
$$ \text{minimize } \begin{pmatrix}Px \\ -e^TPx\end{pmatrix} \text{
            s.t. } x \in W,$$ having $4$ objectives, $729$ variables,
          and $729$ double-sided constraints. The upper image
          intersected with the hyperplane $\cb{y \in \R^4 \st e^T y =
            0}$ is the set $P[W]$, see Figure \ref{fig:2}.
\end{example}	

\begin{figure}[ht]	
	\begin{center}
	\includegraphics[scale=0.3]{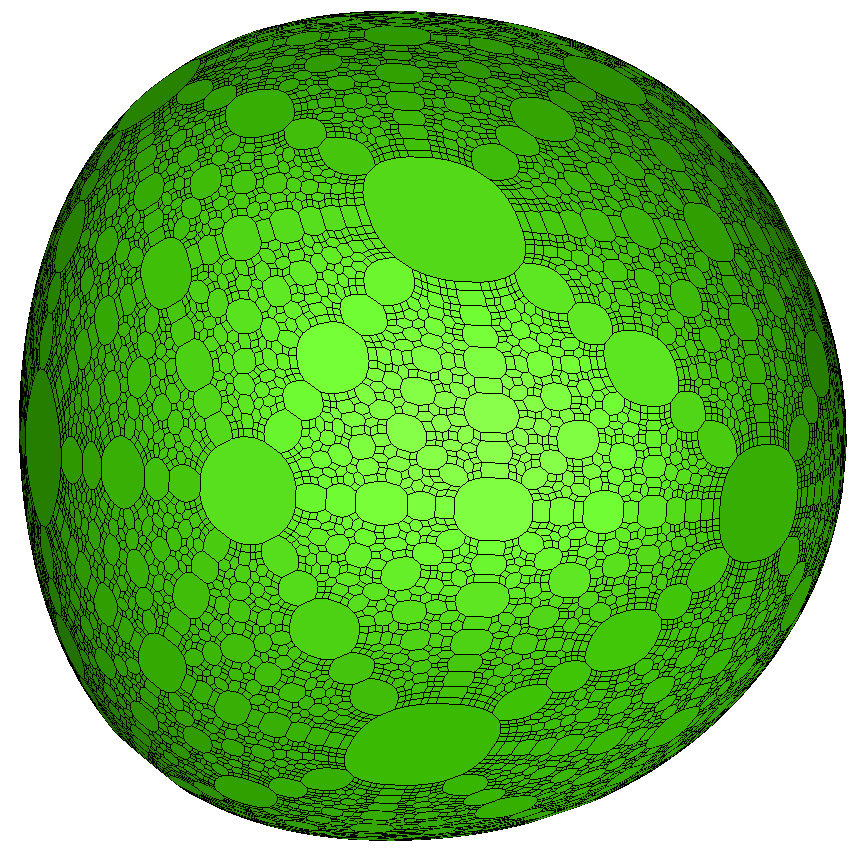}
	\end{center}	
	\caption{The polytope $P[W]$ of Example \ref{ex1} computed by
          {\sc Bensolve} \cite{bensolve,LoeWei16} via MOLP reformulation. The
          resulting polytope has $43680$ vertices and $26186$
          facets. The upper image $\P\subseteq \R^4$ of the
          corresponding MOLP has $43680$ vertices and $26187$
          facets. The displayed polytope is one of these facets, the
          only one that is bounded.}\label{fig:2}
\end{figure}

We close this article by enumerating some related results which can be
found in the literature. It is known \cite{JonKerMac08} that the parametric linear program
\begin{equation}
	\tag{PLP($y$)} \min c^T x \quad \text{ s.t. } \quad G x + H y
        \geq b
\end{equation}
is equivalent to the problem to project a polyhedral convex set onto a
subspace. For every parameter vector $y$ we can consider the dual
parametric linear program
\begin{equation}
	\tag{DPLP($y$)} \max (b-Hy)^T u \quad \text{ s.t. } \quad G^T
        u = c, \; u \geq 0,
\end{equation}
which is closely related to an equivalent characterization of a vector
linear program (or multiple objective linear program) by the family of
all weighted sum scalarizations, see e.g.\ \cite{Focke73}.

F{\"u}l{\"o}p's seminal paper \cite{Fulop93} has to be mentioned
because a problem similar to \eqref{associated_molp} was used there to
show that linear bilevel programming is equivalent to optimizing a
linear objective function over the solution of a multiple objective
linear program.

The book \cite{LotBusKam04} provides interesting links between
multiple objective linear programming and computation of convex
polyhedra.
\end{section}

\bibliographystyle{abbrv}

\end{document}